\documentclass[journal]{IEEEtran}
\usepackage{amssymb,epsfig,color,cite,amsmath,amsfonts,mathrsfs,algorithm,algorithmicx}
\usepackage{epstopdf}
\usepackage{datetime,fancyhdr}
\usepackage{algpseudocode}
\usepackage{arydshln}
\usepackage{multirow}
\usepackage{amsthm} 
\usepackage{hyperref}

\usepackage{times} 
\usepackage{epsfig}
\usepackage{graphicx}
\usepackage{latexsym}
\usepackage{subfigure}
\usepackage{stfloats}

\newtheorem{lemma}{Lemma}[section]
\newtheorem{theorem}{Theorem}[section]
\newtheorem{corollary}{Corollary}[section]
\newtheorem{remark}{Remark}[section]
\newtheorem{definition}{Definition}[section]

\newtheorem{proposition}{Proposition}[section]
\newtheorem{assumption}{Assumption}[section]

\ifCLASSINFOpdf
\else
\fi
\begin{document}
%
\title{Distributed stochastic gradient tracking algorithm with variance reduction for non-convex optimization }
%
%
%

\author{Xia Jiang,
        Xianlin~Zeng,~\IEEEmembership{Member,~IEEE,}
        Jian~Sun,~\IEEEmembership{Member,~IEEE,}
        and~Jie~Chen,~\IEEEmembership{Fellow,~IEEE,}
\thanks{This work was supported in part by the National Natural Science Foundation of China (Nos. 61720106011, 62088101, U20B2073, 61925303), the National Key Research and Development Program of China under Grant 2018YFB1700100 and Beijing Institute of Technology Research Fund Program for Young Scholars. \emph{(Corresponding author: Xianlin Zeng.)}}
\thanks{X. Jiang (jiangxia@bit.edu.cn) and J. Sun (sunjian@bit.edu.cn) are with Key Laboratory of Intelligent Control and Decision of Complex Systems, School of Automation, Beijing Institute of Technology, Beijing, 100081, China, and also with the Beijing Institute of Technology Chongqing Innovation Center, Chongqing  401120, China}
\thanks{X. Zeng (xianlin.zeng@bit.edu.cn) is with Key Laboratory of Intelligent Control and Decision of Complex Systems, School of Automation, Beijing Institute of Technology, Beijing, 100081, China}
\thanks{J. Chen (chenjie@bit.edu.cn) is with Beijing Advanced Innovation Center for Intelligent Robots and Systems (Beijing Institute of Technology), Key Laboratory of Biomimetic Robots and Systems (Beijing Institute of Technology), Ministry of Education, Beijing, 100081, China, and also with the School of Electronic and Information Engineering, Tongji University, Shanghai, 200082, China}}
\maketitle

\begin{abstract}
 This paper proposes a distributed stochastic algorithm with variance reduction for general smooth non-convex finite-sum optimization, which has wide applications in signal processing and machine learning communities. In distributed setting, large number of samples are allocated to multiple agents in the network. Each agent computes local stochastic gradient and communicates with its neighbors to seek for the global optimum. In this paper, we develop a modified variance reduction technique to deal with the variance introduced by stochastic gradients. Combining gradient tracking and variance reduction techniques, this paper proposes a distributed stochastic algorithm, GT-VR, to solve large-scale non-convex finite-sum optimization over multi-agent networks. A complete and rigorous proof shows that the GT-VR algorithm converges to first-order stationary points with $O(\frac{1}{k})$ convergence rate. In addition, we provide the complexity analysis of the proposed algorithm. Compared with some existing first-order methods, the proposed algorithm has a lower $\mathcal{O}(PM\epsilon^{-1})$ gradient complexity under some mild condition. By comparing state-of-the-art algorithms and GT-VR in experimental simulations, we verify the efficiency of the proposed algorithm.
\end{abstract}


%
\IEEEpeerreviewmaketitle

\section{Introduction}
%
%
%
%

\IEEEPARstart{W}{ith} the development of big data, distributed finite-sum optimization has received extensive attention from researchers in signal processing, control and machine learning communities \cite{distri_cs,dis_opti_cy,dis_mides_CIS,finite_sum_ll,decen_op,distri_newton}. In distributed finite-sum optimization, large-scale signal information or training samples are allocated to different nodes, and each node updates variable by local data and obtains global optimal estimate through communication with neighbors. When the dataset is large or is naturally located in a decentralized setting or contains private information, it is infeasible to transmit the dataset over networks and handle it at a centralized node \cite{priv_subg,priv_time_varying,mul_opti,Dis_coop_CIS,finite_sum_ll}. In addition, for functions with large-size local data, computing the local gradient by the entire local dataset becomes practically difficult. Hence, methods that sample small batches of local dataset to compute stochastic gradients are favored. Due to the above reasons, distributed stochastic first-order algorithms are preferable as they own a low computation complexity without the calculation of Hessian matrix and are easy to analyze.
\par Distributed stochastic gradient algorithm is a combination of average consensus steps between neighbors and local stochastic gradients, which has been popular in machine learning tasks \cite{distri_g_ll,distri_ll_sg,strong_sg_ll}. With the similar design idea, considerable works have been studied for more complex optimization problems and various multi-agent networks in recent years, e.g., distributed stochastic gradient projection algorithms \cite{2010Distributed}, distributed stochastic mirror descent \cite{mirror_des},  distributed stochastic primal-dual algorithm over random networks with imperfect communications \cite{Lei2018ASYMPTOTIC} and stochastic gradient-push over time-varying directed graphs \cite{Nedic2013Distributed}. However, the performance of distributed stochastic gradient algorithm is generally limited by two components. One is the local variance introduced by stochastic gradient at each agent and the other is the heterogeneous datasets between different agents. To handle the local variance, many variance reduction techniques have been proposed to reduce storage space and computation complexity, such as SAGA \cite{saga_pa}, SVRG \cite{SVRG_pa}, SARAH \cite{SARAH_pa} and Asyn-VR \cite{asyn_vr}. Distributed variance-reduced stochastic gradient methods have also been developed for smooth and strongly-convex optimization in recent years \cite{gt-vr-strong,dis_stoch,linear_vr}. To achieve robustness to heterogeneous environments, some works develop distributed bias-correction techniques such as gradient tracking \cite{next_pa,linear_con_pa}, EXTRA \cite{extra_weishi}, and primal-dual principles \cite{DLAforcons,Ling2015DLMDL}. Integrating variance reduction and bias-correction techniques, efficient distributed algorithms with linear convergence rate arise for strongly-convex finite-sum optimization \cite{gt-vr-strong}. However, the applicability of these distributed methods for non-convex optimization remains unclear.
\par Large-scale non-convex optimization has wide applications including logistic regression with non-convex regularization and neural networks training. When the cost functions are non-convex, the design and theoretical analysis of efficient algorithms become difficult due to the lack of good properties of convexity. 
Very recent works have proposed distributed variance-reduced methods for non-convex finite-sum problems. \cite{sun2019improving} has proposed D-GET for decentralized non-convex finite-sum minimization, which considers local SARAH-type variance-reduced technique and gradient tracking. However, as is pointed by \cite{xin2020nearoptimal}, D-GET does not have a network-independent gradient complexity. 
GT-SARAH proposed in \cite{xin2020nearoptimal} has achieved a near-optimal total gradient computation complexity at the cost of twice communication rounds of the D-GET algorithm.
\cite{xin2020fast} has proposed a GT-SAGA algorithm by combining the SAGA and gradient tracking techniques. However, SAGA needs additional storage space compared with SVRG. Inspired by SVRG technique, we propose a distributed stochastic first-order algorithm with low network-independent gradient complexity for large-scale finite-sum non-convex optimization in this paper. 
\par The contributions of this paper are summarized as following.
\begin{itemize}
	\item [(1)]  For general smooth non-convex optimization, we propose a novel distributed stochastic iterative algorithm, GT-VR, by combining gradient tracking and variance reduction techniques. The variance reduction in the proposed algorithm is a modified version of SVRG technique \cite{SVRG_pa} and makes use of Bernoulli distribution to reduce the local variance introduced by stochastic gradient.
	\item [(2)] By linear matrix inequality, we prove that the proposed GT-VR algorithm converges to first-order stationary points with $O(\frac{1}{k})$ convergence rate. To the best of our knowledge, for general smooth non-convex optimization, it is the first work to provide a sublinear convergence rate without steady-state error for distributed stochastic algorithms designed by gradient tracking and variance reduction. In addition, we provide the range of constant step-sizes and the probability range of the Bernoulli distribution.
	\item [(3)] Compared with some newest algorithms, the proposed algorithm has a lower gradient complexity. To be specific, compared with distributed algorithms DSGT \cite{DSGT}, D2 \cite{pmlr-v80-tang18a}, DSGD \cite{DSGD}, whose gradient complexity is $\mathcal{O}(\nu^2\epsilon^{-2})$, the proposed algorithm has a lower network-independent gradient complexity $\mathcal{O}(PM\epsilon^{-1})$ under some mild condition. Comparative experimental results of these algorithms and GT-VR also verify the efficiency of the proposed algorithm. 
\end{itemize}

 \par The remainder of the paper is organized as follows. Mathematical notations and some stochastic properties are given in section \ref{preliminaries_sec}. The problem description and distributed stochastic algorithm are provided in section \ref{solver_design}. The convergence properties of proposed methods are provided in section \ref{conver_sec} and are analyzed theoretically in section \ref{analysis_sec}. The efficiency of the distributed algorithms are verified by simulations in Section \ref{simulation} and the conclusion is made in section \ref{conclusion}. 

\section{Notations and Preliminaries} \label{preliminaries_sec}
\subsection{Mathematical notations}
\par We denote $\mathbb{R}$ as the set of real numbers, $\mathbb{R}^+$ as the set of positive real numbers, $\mathbb{Z}^+$ as the set of positive integers, $\mathbb{R}^n$ as the set of $n$-dimensional real column vectors, respectively. All vectors in the paper are column vectors, unless otherwise noted. $\mathbf{1}_n$ denotes an $n \times 1$ vector with all elements of $1$, $0_d$ denotes a $d\times 1$ vector with all elements of $0$ and $I_d$ denotes a $d\times d$ identity matrix. The notation $\otimes$ denotes the Kronecker product and ${\rm max}\{\cdots\}$ denotes the maximum element in the set $\{\cdots\}$. For a real vector $v$, $\left\|v\right\|$ is the Euclidean norm. For a differentiable function $f(x)$, its gradient is represented by $\nabla f(x)$. In the following paper, subscript $i$ refers to this being the local variables of the $i$th agent, e.g., $x_i$ means local variable $x$ of agent $i$.
\par We fix a rich enough probability space $(\Omega,\mathcal{F},\mathbb{P})$, where all random variables in discussion are properly defined and $\mathbb{E}[\cdot]$ denotes the expectation operator with respect to the probability measure $\mathbb{P}$. Let $A$ be an event in $\mathcal{F}$ with nonzero probability and $X$ be a random variable. The conditional expectation of $X$ given $A$ is denoted by $\mathbb{E}[X|A]$. For an event $A\in \mathcal{F}$, its indicator function is denoted as $\mathbb{I}_A$. We use $\sigma(\cdot)$ to denote the $\sigma-algebra$ generated by the random variables and/or sets in its argument. For a matrix $A$, $d(A)$ denotes its spectral radius.

\subsection{Stochastic Theory}
For conditional expectation, there is one basic property, which is useful in the subsequent analysis and is stated as following.   
\begin{proposition}
	Let $X$, $Y$ and $X_i(1\leq i\leq n)$ be random variables, and $\mathbb{E}|X|<\infty$, $\mathbb{E}|X_i|<\infty(1\leq i\leq n)$ .
	\begin{itemize}
		\item [(1)] $\mathbb{E}[\sum_{i=1}^n \alpha_i X_i |Y]=\sum_{i=1}^n \alpha_i \mathbb{E}[X_i|Y]$ $(a.s.)$, where $\alpha_i$ is a constant.
		\item[(2)] If $X$ and $Y$ are mutual independent, then $\mathbb{E}[X|Y]=\mathbb{E}X$.
		\item[(3)]  $\mathbb{E}[\mathbb{E}[X|Y]]=\mathbb{E}X$.
	\end{itemize} 
\end{proposition}
\par \textbf{Bernoulli distribution}: The Bernoulli($P$) distribution is the discrete probability distribution of a random variable which takes the value $1$ with probability $P$ and the value $0$ with probability $1-P$. If $X$ is a random variable with the Bernoulli($P$) distribution, then
\begin{align}
\mathbf{Pr}(X=1)=P, \quad \mathbf{Pr}(X=0)=1-P.
\end{align}

\section{Problem Description and Distributed Solver Design}\label{solver_design}
In this paper, we aim to solve the following distributed finite-sum optimization problem over a multi-agent network,
\begin{align}\label{pro_f}
\min_{x\in\mathbb{R}^d} f(x)\triangleq \frac{1}{n}\sum_{i=1}^n f_i(x), \ f_i(x)=\frac{1}{m_i}\sum_{j=1}^{m_i}f_{i,j}(x),
\end{align}
where $f_i:\mathbb{R}^d\to \mathbb{R}$ is the local differentiable objective function of agent $i$, further decomposed as the average of $m_i$ component costs $\{f_{i,j}\}_{j=1}^{m_i}$, $n$ is the number of agents, $m_i$ is the number of local samples, and $\sum_{i=1}^n m_i=M$ is the total number of samples in the network. The multi-agent network containing $n$ agents is denoted by $\mathcal{G}(\mathcal{V},\mathcal{E}, W)$, where $\mathcal{V}=\{1,\cdots,n\}$, $\mathcal{E}=\mathcal{V}\times \mathcal{V}$ and $W$ is the adjacent matrix associated with $\mathcal{G}$. In distributed setting, each agent handles local information and communicates with its neighbors over the network $\mathcal{G}$ to solve \eqref{pro_f} cooperatively.
\begin{remark}
	This formulation of optimization problem \eqref{pro_f} is widely adopted in empirical risk minimization, where each local cost $f_i$ can be considered as an empirical risk computed over a finite number of $m_i$ local data samples and lies at the heart of many modern machine learning problems \cite{mini_pro,sto_cg}. Compared with the parameter-server type machine learning system with a fusion center \cite{mach_learn_overair,training_reduc_commu}, distributed optimization problem \eqref{pro_f} can preserve data privacy, improve the computation efficiency and enhance network robustness. Furthermore, in many emerging applications such as collaborative filtering, federated learning, distributed beamforming and dictionary learning, the data is naturally collected in a distributed setting, and it is impossible to transfer the distributed data to a central location \cite{sun2019improving}. Therefore, decentralized computation has sparked considerable interest in both academia and industry. 
\end{remark}

\par Next, we design a distributed stochastic algorithm for the general smooth non-convex optimization \eqref{pro_f}. There are two challenges in the distributed stochastic algorithm design. One challenge is the slow convergence due to the variance of stochastic gradients by asymptotically estimating the local full gradient $\nabla f_i$, based on randomly selected samples from the local dataset of agent $i$. The other is the difference between local and global objective functions, i.e., $\nabla f_i(x^*)\neq 0_d$, $\forall i \in \{1,\cdots, n\}$, holds for the global optimum $x^*$. This issue may be handled by the popular gradient tracking technique that introduces a local gradient estimator to track the global gradient.
\par By combining the distributed gradient tracking \cite{zero_gt} with a variance reduction technique, we propose a first-order GT-VR algorithm. The complete implementation of GT-VR is summarized in Algorithm \ref{algo_sum}. The local gradient estimator $v_i^k$ is updated by 
\begin{align}\label{v_update}
	v_i^{k}=\nabla f_{i, s_{i}^{k}}(x_i^{k})-\nabla f_{i, s_{i}^{k}}(\tau_i^{k})+\nabla f_{i}(\tau_i^{k}).
\end{align} 
In addition, in GT-VR, we introduce the gradient tracking technique to achieve the global gradient tracking in distributed optimization.
\begin{algorithm}[H]
\caption{ GT-VR updating at each agent $i$}
	\label{algo_sum}
	\begin{algorithmic}[1]  
        \State Initialize: $x_i^1$; $\tau_i^1=x_i^1$; $\eta$; $\{w_{ir}\}_{r=1}^n$; $y_i^1=v_i^1=\nabla f_i(x_i^1)$.
        \For {$k=1,2,\cdots$}
        \State Update the local estimate of the solution:
        $$x_i^{k+1}=\sum_{r=1}^n w_{ir}(x_r^k-\eta y_r^k);$$
        
        \State Select $l_i^{k+1}$ at random from the Bernoulli($P$) distribution. If $l_i^{k+1}=1$,  $\tau_i^{k+1}=x_i^{k+1}$, and otherwise, $\tau_i^{k+1}=\tau_i^k$.
        \State Select $s_i^{k+1}$ uniformly at random from $\{1,\cdots,m_i\}$.
        \State Update the local stochastic gradient estimator by \eqref{v_update};
        \State Update the local gradient tracker:
        $$y_i^{k+1}=\sum_{r=1}^n w_{ir}(y_r^{k}+v_r^{k+1}-v_r^k);$$
        \EndFor
	\end{algorithmic}
\end{algorithm}
\begin{remark}
	 The variance reduction technique taken in GT-VR is a modification of the well-known SVRG technique. In both techniques, the entire local full gradient $\nabla f_i(x_i^k)$ needs to be computed with a certain probability and the local gradient estimator $v_i^k$ is updated by \eqref{v_update}. The only main difference is that in GT-VR, $\tau_i^k$ is updated following the Bernoulli distribution, while in SVRG, it is updated periodically. Denote $\mathcal{F}^k$ as the history of the dynamical system defined by $\sigma(\{s_i^t,l_i^t\}_{i=\{1,\cdots,n\}}^{t\leq k-1})$. Note that in SVRG, each local gradient estimator $v_i^k$ is an unbiased estimator of the local gradient $\nabla f_i(x_i^k)$ given $\mathcal{F}^k$ \cite{SVRG_pa}, whereas, it does not hold in our proposed algorithm GT-VR. This modification is vital for the convergence of proposed algorithm for the smooth non-convex optimization.
\end{remark}
\begin{remark}
	\par Compared with distributed deterministic optimization, which needs to compute the entire local gradient $\nabla f_i$  at each iteration, the proposed distributed stochastic first-order algorithm using sampled batch data to compute stochastic gradient is more suitable for the training and processing of large-scale data.
	
	\par Compared with GT-SAGA \cite{xin2020fast}, the proposed algorithm does not need to store the value of gradient and saves more storage space.
	Compared with the two-timescale hybrid algorithm GT-SARAH \cite{xin2020nearoptimal}, the proposed algorithm is one single-timescale randomized gradient algorithm. In addition, at each iteration, there are only two communication rounds with neighbors at each agent. However, GT-SARAH has a near-optimal gradient computational complexity, which is better than the proposed algorithm. 
\end{remark}
\section{Convergence Result}\label{conver_sec}
In this section, we provide the convergence analysis of the proposed algorithm GT-VR with some mild assumptions.
\begin{assumption}\label{W_assump}
	\begin{itemize}
		\item[(1)] Each cost function $f_{i,j}$ is uniformly $L$-smooth, i.e., for some $L>0$,
		$$\|\nabla f_{i,j}(x)-\nabla f_{i,j}(y)\|\leq L\|x-y\|, \ \forall x,y \in \mathbb{R}^d.$$
		\item[(2)] The adjacent matrix of $\mathcal{G}$, $W$, is a doubly stochastic matrix, where $W_{ii}>0$ for all $i\in \mathcal{V}$, and $W_{ij}>0$ if $(i,j)\in \mathcal{E}$ for $i, j \in \mathcal{V}$.
		\item[(3)] The family $\{l_i^k,s_i^k: i\in \mathcal{V}, k\geq 1\}$ of random variables in the proposed algorithm is independent.
	\end{itemize}
\end{assumption}
\par Assumption \ref{W_assump} (1) and (2) are common assumptions in distributed optimization. (1) guarantees that local batch objective functions $\{f_i\}_{i=1}^n$ and the global objective function $f$ are $L$-smooth. The adjacent matrix satisfying (2) holds for the family of undirected graphs and weight-balanced directed graphs. In addition, Assumption \ref{W_assump} (2) guarantees that the radius of the network $\rho$ satisfies
\begin{align}
	\rho\triangleq \sup_{\|x=1\|} \frac{\|W(x-n^{-1} \mathbf{1}_n \mathbf{1}_n^Tx)\|}{\|x-n^{-1} \mathbf{1}_n \mathbf{1}_n^Tx\|}<1.
\end{align}
Assumption \ref{W_assump} (3) is standard in the design of stochastic algorithm and is practical in application. 

 For the convenience of analysis, we define several  auxiliary quantities as following:
$$
{\bf x}^k =\left[\begin{matrix}x_1^k\\
\vdots\\
x_n^k
\end{matrix}\right],\ {\bf y}^k=\left[\begin{matrix}y_1^k\\ \vdots \\ y_n^k\end{matrix}\right], \  {\bf v}^k=\left[\begin{matrix}v_1^k\\ \vdots\\ v_n^k\end{matrix}\right],$$
$$ {\tau}^k=\left[\begin{matrix}\tau_1^k\\ \vdots\\ \tau_n^k\end{matrix}\right],
\ \nabla f(\mathbf{x}^k)=\left[\begin{matrix} \nabla f_1(x_1^k)\\ \vdots\\ \nabla f_n(x_n^k)\end{matrix}\right],$$\\
 $\bar{x}^k=\frac{1}{n}\sum_{i=1}^n x_i^k$, $\bar{y}^k=\frac{1}{n}\sum_{i=1}^n y_i^k$, $\bar{v}^k=\frac{1}{n}\sum_{i=1}^n v_i^k$, $\mathbf{\bar{x}}^k=\mathbf{1}_n\otimes \bar{x}^k\in\mathbb{R}^{nd}$ and $\mathbf{\bar{v}}^k=\mathbf{1}_n\otimes \bar{v}^k\in\mathbb{R}^{nd}$.
 \par Then, the proposed algorithm GT-VR satisfies
 \begin{align}
 {\bf x}^{k+1}&=(W\otimes I_d) \big({\bf x}^k-\eta {\bf y}^k\big),\label{x_up}\\
 {\bf y}^{k+1}&=(W\otimes I_d) \big({\bf y}^k+{\bf v}^{k+1}-{\bf v}^k\big),\label{y_up}
 \end{align}
 and 
 \begin{align}\label{bareq}
 \bar{y}^k=\bar{v}^k,\qquad \bar{x}^{k+1}=\bar{x}^k-\eta \bar{y}^k,
 \end{align}
 where the doubly stochastic property of $W$ is used to derive \eqref{bareq}.
 \par Define $\bar{\eta}=\min \Big\{ \frac{(1-3\rho^2)}{\big(16\rho^2 L^2+(32\rho^2 L^2+2)(1-P)\frac{\rho+1}{\rho}\big)5L},\frac{1}{6L}, \\ \sqrt{\frac{1-(\frac{4}{3}+\frac{8}{9} P)\rho^2}{2\mathrm{T}}} \Big\}$, where $\mathrm{T} \triangleq 16L^2+\big(\frac{8}{3}+\frac{16}{3}(1+\frac{1}{\rho})(1-P)\big)L^2+ (32+32P )L^2\rho^2 + \big(\frac{16}{9}+16(1-P)(1+\rho +\frac{2(\rho+1)}{9\rho})\big)L^2 \varepsilon_3\in \mathbb{R}^+$ and $\varepsilon_3$ is a positive number. 
\par Then, the convergence result of the proposed algorithm GT-VR is covered in the following theorem.
\begin{theorem}\label{conver_theo}
	Suppose Assumption \ref{W_assump} hold. Let $\rho^2<\frac{1}{3}$, $1-\frac{3\rho^2}{(1+\frac{1}{\rho})\frac{2}{9}+1+\rho}<P<1$, $0<\eta < \bar{\eta}$ and $f^*\triangleq \inf_{x\in\mathbb{R}^d} f(x)>-\infty$.\\ 
	Then $f(\bar{x}^k)$ converges,		$\frac{1}{k}\sum_{s=1}^{k}\mathbb{E}\|\nabla f(\bar{x}^{s})\|^2\leq \mathcal{O}(\frac{1}{k})$, $\frac{1}{k}\sum_{s=1}^{k}\mathbb{E}\|{\bf x}^{s}-{\bf \bar{x}}^{s}\|^2\leq  \mathcal{O}(\frac{1}{k}) $ and $\frac{1}{k}\sum_{s=1}^{k}\mathbb{E}\|{\bf y}^{s}-\nabla f(\mathbf{\bar{x}}^{s})\|^2\leq \mathcal{O}(\frac{1}{k})$.

\end{theorem}
\begin{remark}
	Theorem \ref{conver_theo} implies that the proposed GT-VR algorithm converges to first-order stationary points with $O(\frac{1}{k})$ convergence rate. Compared with GT-HSGD \cite{xin2021hybrid}, which converges sublinearly at a rate of $O(\frac{1}{k})$ up to a steady-state error, the proposed GT-VR has no steady-state error.
\end{remark}
\par Then, we present the complexity of GT-VR in the following sense.
\begin{definition}\label{acc_def}
	The algorithm GT-VR is said to achieve an $\epsilon$-accurate stationary point of $f$ in $k$ iterations if
	\begin{align}
		\frac{1}{k}\sum_{s=1}^k \mathbb{E}[\|\nabla f(\bar{x}^s)\|^2]\leq \epsilon.
	\end{align}
\end{definition}
Based on the results in Theorem \ref{conver_theo}, the iteration complexity, gradient computation complexity and communication complexity of GT-VR are established in the following corollary.
\begin{corollary}\label{conver_corol}
	Suppose Assumption \ref{W_assump} hold. Let $\rho^2<\frac{1}{3}$, $1-\frac{3\rho^2}{(1+\frac{1}{\rho})\frac{2}{9}+1+\rho}<P<1$ and
	\begin{align}\label{eta_acc}
		0<\eta \leq \tilde{\eta},
	\end{align}
where $\tilde{\eta}=\min\big\{\bar{\eta},\frac{1-3\rho^2}{3\rho^2L}\big\}$.
	Then,  
	\begin{itemize}
		\item [(1)] GT-VR achieves an $\epsilon$-accurate stationary point of $f$ with $\mathcal{O}(\eta^{-2} \epsilon^{-1})$ iterations.
		\item [(2)] GT-VR achieves an $\epsilon$-accurate stationary point of $f$ in $\mathcal{O}(PM\eta^{-2} \epsilon^{-1})$ gradient computations across all agents.
		\item [(3)] GT-VR achieves an $\epsilon$-accurate stationary point of $f$ with $\mathcal{O}(\sum_{i=1}^n N_i\eta^{-2} \epsilon^{-1})$ communication rounds across all agents, where $N_i$ denotes the number of neighbors of agent $i$.
	\end{itemize}
\end{corollary}
\begin{remark}
	 In large-scale sample case, the gradient complexity is $\mathcal{O}(PM\eta^{-2}\epsilon^{-1})$. If $\eta=\frac{1}{6L}$,
	the gradient complexity is $\mathcal{O}(PM\epsilon^{-1})$, which is network-independent. If the accuracy $\epsilon$ is chosen small enough, this gradient complexity is lower than the gradient complexity $\mathcal{O}(\nu^2\epsilon^{-2})$ of DSGT \cite{DSGT}, D2 \cite{pmlr-v80-tang18a} and DSGD \cite{DSGD}, where $\nu$ is the bounded variance of stochastic gradient within each agent. 
\end{remark}

\section{theoretical analysis}\label{analysis_sec}
\par In this section, we present the proofs for Theorem \ref{conver_theo} and Corollary \ref{conver_corol}. The analysis framework is general and may be applied to other distributed algorithms based on variance reduction and gradient tracking. Recall that $\mathcal{F}^k$ is the history of the dynamical system, defined by $\mathcal{F}^k\triangleq \sigma(\{s_i^t,l_i^t\}_{i=\{1,\cdots,n\}}^{t\leq k-1})$. The convergence analysis is roughly divided into three steps. (1) The first step is to prove the boundness of $\mathbb{E}[\|{\bf x}^k-{\bf \bar{x}}^k\|^2]$, $\mathbb{E}[\|{\bf \tau}^k-{\bf \bar{x}}^k\|^2]$ and
$\mathbb{E}[\|{\bf y}^k-{\bf \bar{v}}^k\|^2]$ by constructing a linear matrix inequality. (2) The second step is to prove the boundness of  $\sum_{s=1}^k \mathbb{E}[\|{\bf x}^s-{\bf \bar{x}}^s\|^2]$, $\sum_{s=1}^k \mathbb{E}[\|{\bf \tau}^s-{\bf \bar{x}}^s\|^2]$ and
$\sum_{s=1}^k \mathbb{E}[\|{\bf y}^s-{\bf \bar{v}}^s\|^2]$ by recursion. (3) Finally, we prove the boundness of $\sum_{s=1}^k \mathbb{E}[\|\nabla f (\bar{x}^s)\|^2]$ and the convergence of $f(\bar{x}^k)$.

\par At first, we provide a standard result for the adjacent matrix $W$ satisfying Assumption \ref{W_assump}, which will be frequently used in the subsequent analysis.
\begin{lemma}\cite{zero_gt}\label{W_lemma}
	Suppose Assumption \ref{W_assump} hold. For any $x_1,\cdots,x_n\in \mathbb{R}^d$, we have
	\begin{align*}
		\|(W\otimes I_d)(\mathbf{x}-\mathbf{\bar{x}})\|\leq \rho \|\mathbf{x}-\mathbf{\bar{x}}\|,
	\end{align*}
	where $\rho$ is the radius of the underlying network.
\end{lemma}
\par In the following proposition, we present some useful properties of local stochastic gradient estimator $v_i^k$, local gradient $\nabla f_i(x_i^k)$ and $\nabla f_i(\bar{x}^k)$, which will be used to bound  $ \mathbb{E}[\|{\bf \tau}^k-{\bf \bar{x}}^k\|^2]$ and to prove Theorem \ref{conver_theo}.

\begin{proposition}\label{vnablaf}
	Suppose Assumption \ref{W_assump} holds. Then,
		\begin{align}
			 &\mathbb{E}[\|\mathbf{v}^k-\nabla f(\mathbf{x}^k)\|^2]\notag\\
			\leq& 2L^2\mathbb{E}[\|\mathbf{x}^k-\bar{\mathbf{x}}^k\|^2]+2L^2\mathbb{E}[\|\mathbf{\tau}^k-\bar{\mathbf{x}}^k\|^2],
		\end{align}
		\begin{align}\label{vsubfcondi}
			&\mathbb{E}[\|\frac{1}{n}\sum_{i=1}^n(v_i^k-\nabla f_i(\bar{x}^k))\|^2]\notag\\
			\leq& \frac{1}{n}\sum_{i=1}^n(6L^2\mathbb{E}[\|x_i^k-\bar{x}^k\|^2]+4L^2\mathbb{E}[\|\tau_i^k-\bar{x}^k\|^2])
		\end{align}
	 and
		\begin{align}\label{vcondi}
			&\frac{1}{2}\mathbb{E}[\|\bar{v}^k\|^2]
			\leq \mathbb{E}[\|\nabla f(\bar{x}^k)\|^2]\notag\\
			&+\frac{2L^2}{n}\sum_{i=1}^n(3\mathbb{E}[\|x_i^k-\bar{x}^k\|^2]+2\mathbb{E}[\|\tau_i^k-\bar{x}^k\|^2]).
		\end{align}
\end{proposition}
\begin{proof}
	See Appendix \ref{vnablaproof}.
\end{proof}
 \par With Proposition \ref{vnablaf}, in the following lemma, we establish bounds on $\mathbb{E}[\|{\bf x}^k-{\bf \bar{x}}^k\|^2]$, $\mathbb{E}[\|\tau^k-{\bf \bar{x}}^k\|^2]$ and
 $\mathbb{E}[\|{\bf y}^k-{\bf \bar{v}}^k\|^2]$, respectively.
 \begin{proposition}\label{xsub}
 	Suppose Assumption \ref{W_assump} hold. We have the following inequalities:
 		\begin{align}
 		{\rm (a)}\quad	&\mathbb{E}[\|{\bf x}^{k+1}-{\bf \bar{x}}^{k+1}\|^2]\notag\\
 			\leq &2\rho^2\mathbb{E}[\|{\bf x}^k-{\bf \bar{x}}^k\|^2]+2\rho^2\eta^2\mathbb{E}[\|{\bf y}^k-{\bf \bar{v}}^k\|^2]\label{xsubxbareq},
 		\end{align}
 		\begin{align}
 		{\rm (b)}\quad	&\mathbb{E}[\|\mathbf{\bar{x}}^{k+1}-\tau^{k+1}\|^2]\notag\\
 			\leq & \big(2\rho^2 P +(1-P)(\eta^2+\frac{\eta}{\beta})12L^2\big) \mathbb{E}[\|\mathbf{\bar{x}}^k-\mathbf{x}^k\|^2]\notag\\
 			& + 2\rho^2 \eta^2 P \mathbb{E}[\|\mathbf{y}^k-\mathbf{\bar{v}}^k\|^2]\notag\\
 			&+(1\!-\!P)\big((\eta^2+\frac{\eta}{\beta})8L^2+(1+\eta\beta)\big) \mathbb{E}[\|\mathbf{\bar{x}}^k -\tau^k\|^2]\notag\\
 			&+ (\eta^2+\frac{\eta}{\beta}) 2n(1-P) \mathbb{E}[\|\nabla f(\bar{x}^k)\|^2],\label{tausubxbar}
 		\end{align}
 		\begin{align}
 		{\rm (c)}\quad	&\mathbb{E}[\|{\bf y}^{k+1}-{\bf \bar{v}}^{k+1}\|^2]\notag\\
 			\leq &(2\rho^2+48\eta^2L^2\rho^4+32 P \eta^2L^2\rho^4) \mathbb{E}[\|{\bf y}^k-{\bf \bar{v}}^k\|^2]\notag\\
 			&+16\rho^2L^2\big(1+6\eta^2L^2+(2+2P)\rho^2\notag\\
 			&\qquad+12L^2(1-P)(\eta^2+\frac{\eta}{\beta})\big)\mathbb{E}[\|{\bf x}^k-{\bf \bar{x}}^k\|^2]\notag\\
 			&+16\rho^2L^2\big(4\eta^2L^2\!+\!(1\!-\! P)[1+\eta\beta + (\eta^2+\frac{\eta}{\beta})8L^2]\big)\notag\\
 			&\qquad\mathbb{E}[\|\tau^k-{\bf \bar{x}}^k\|^2]\notag\\
 			&+16\rho^2L^2n\big(\eta^2+2 (1-P)(\eta^2+\frac{\eta}{\beta})\big)  \mathbb{E}[\|\nabla f(\bar{x}^k)\|^2],\label{ysub}
 		\end{align}
 where $\beta \in \mathbb{R}^+$, $\eta$ is the step-size and $\rho$ is the radius of the underlying network.
 \end{proposition}
\begin{proof}
	See Appendix \ref{x_proof}.
\end{proof}
\par  With \eqref{xsubxbareq}-\eqref{ysub} in Proposition \ref{xsub} bounded in terms of linear combinations of their past values, we will establish a linear system of inequalities to bound $\sum_{s=1}^k \mathbb{E}[\|{\bf x}^s-{\bf \bar{x}}^s\|^2]$, $\sum_{s=1}^k \mathbb{E}[\|{\bf \tau}^s-{\bf \bar{x}}^s\|^2]$ and
$\sum_{s=1}^k \mathbb{E}[\|{\bf y}^s-{\bf \bar{v}}^s\|^2]$ by recursion.
Before that, we introduce one lemma about the spectral radius of a non-negative matrix, whose proof is in \cite[Theorem 8.3.2]{spectral_book}
\begin{lemma}\label{normA}
	Let $A\in \mathbb{R}^{d\times d}$ be non-negative and $x\in \mathbb{R}^d$ be positive. If $Ax\leq \Theta x$ for $\Theta>0$, then the spectral radius of $A$ satisfies ${d}(A)\leq \Theta$.
\end{lemma}

\begin{proposition}
	Suppose Assumption \ref{W_assump} hold. Suppose  $\rho^2<\frac{1}{3}$, $1-\frac{3\rho^2}{(1+\frac{1}{\rho})\frac{2}{9}+1+\rho}<P<1$ and 
	$$0<\eta <\min \Big\{ \frac{1}{6L},
	\sqrt{\frac{1-(\frac{4}{3}+\frac{8}{9} P)\rho^2}{2\mathrm{T}}} \Big\},$$
	where $\mathrm{T} \triangleq \frac{C_2+C_3\varepsilon_3}{\rho^2}\in \mathbb{R}^+$ and $\varepsilon_3\in \mathbb{R}^+$. 
	We have
	\begin{align}
		&\max\{\sum_{s=1}^{k}\mathbb{E}\|{\bf y}^{s}-{\bf \bar{v}}^{s}\|^2,\sum_{s=1}^{k}\mathbb{E}\|{\bf x}^{s}-{\bf \bar{x}}^{s}\|^2, \sum_{s=1}^{k} \mathbb{E}\|\tau^{s}-{\bf \bar{x}}^{s}\|^2\}\notag\\
		\leq & \frac{3\rho^2 R_0}{1-3\rho^2}+\frac{C_4+C_4''}{1-3\rho^2}\sum_{s=1}^{k}\mathbb{E}\|\nabla f(\bar{x}^{s})\|^2,
	\end{align}
	where $R_0=\mathbb{E}\|{\bf y}^1-{\bf \bar{v}}^1\|^2+\mathbb{E}\|{\bf x}^1-{\bf \bar{x}}^1\|^2+\mathbb{E}\|\tau^1-{\bf \bar{x}}^1\|^2$, $C_4$ and $C_4''$ are defined after \eqref{lmi_ineq}.
\end{proposition}
\begin{proof}
	\par Let $\beta=\frac{\rho}{\eta}$ and $\eta L\leq \frac{1}{6}$. With Lemma \ref{xsub}, we have
	\begin{align}\label{lmi_ineq}
		u_{k+1} 
		\!\leq\! \underbrace{\!\begin{bmatrix}C_1& C_2&C_3\\2\rho^2\eta^2& 2\rho^2& 0\\2\rho^2\eta^2P& C_2'' & C_1''\end{bmatrix}\!}_{C}\! u_k\!+\!\begin{bmatrix}C_4 \\ 0 \\ C_4'' \end{bmatrix} \!\mathbb{E}\|\nabla f(\bar{x}^k)\|^2,
	\end{align}
where $u_k\triangleq \begin{bmatrix}\mathbb{E}\|{\bf y}^{k}-{\bf \bar{v}}^{k}\|^2\\ \mathbb{E}\|{\bf x}^{k}-{\bf \bar{x}}^{k}\|^2\\ \mathbb{E}\|\tau^{k}-{\bf \bar{x}}^{k}\|^2\end{bmatrix}$, $C_1=2\rho^2+\frac{12+8P}{9}\rho^4 $, $C_2=16\rho^2L^2+\frac{8\rho+16(\rho+1) (1-P)}{3\rho}\rho^2L^2+ (32+32P )L^2\rho^4$, $C_3=\frac{16}{9}\big(1+\frac{9\rho^2+11\rho+2}{\rho}(1-P)\big)\rho^2L^2 $, $C_2''=2\rho^2 P +\frac{1}{3} (1-P) (1+\frac{1}{\rho})$, $C_1''=(1-P)\big(\frac{2}{9} (1+\frac{1}{\rho})+1+\rho \big)$
and $C_4=16n \rho^2 \eta^2 L^2+32 n(1-P) \rho^2 \eta^2 L^2(1+\frac{1}{\rho})$ and $C_4''=2 \eta^2 n (1-P)(1+\frac{1}{\rho})$
	\par By induction of \eqref{lmi_ineq}, it leads to
		\begin{align}\label{conv_matrix2}
			u_{k+1}\leq C^k u_1+\sum_{s=0}^{k-1} C^s\begin{bmatrix}C_4 \\ 0 \\ C_4''\end{bmatrix} \mathbb{E}\|\nabla f(\bar{x}^{k-s})\|^2
		\end{align}
		If the spectral radius of $C$ satisfies ${d}(C)<1$, then $C^k$ converges to zero at the linear rate $\mathcal{O}({d}(C)^k)$\cite{spectral_book}. Hence, we next prove $d(C)<1$ by Lemma \ref{normA} so as to prove the result in this proposition. Because $1-\frac{3\rho^2}{(1+\frac{1}{\rho})\frac{2}{9}+1+\rho}<P<1$, we have $\varepsilon_3\triangleq\frac{3\rho^2 P +\frac{1}{3}(1-P)(1+\frac{1}{\rho})}{3\rho^2-(1-P)\big(\frac{2}{9}(1+\frac{1}{\rho})+1+\rho \big)}>0$. Define a positive vector $\varepsilon=[\frac{1}{2\eta^2},1,\varepsilon_3]^T$. Then, with $0<\eta <\min \Big\{ \frac{1}{6L},
		\sqrt{\frac{1-(\frac{4}{3}+\frac{8}{9} P)\rho^2}{2\mathrm{T}}} \Big\}$ and $\rho^2<\frac{1}{3}$, the non-negative matrix $C$ satisfies $C\varepsilon\leq 3\rho^2 \varepsilon$, i.e.
	\begin{subequations}
		\begin{align}
			& (2\rho^2+\frac{12+8P}{9}\rho^4)\frac{1}{2\eta^2}  +C_2 +C_3\varepsilon_3 \leq  \frac{3 \rho^2}{2\eta^2},\label{v1}\\
			&2\rho^2 \eta^2 \frac{1}{2\eta^2} +2\rho^2 = 3\rho^2, \label{v2}\\
			& 2\rho^2 \eta^2 P\frac{1}{2\eta^2}+ ( 2\rho^2 P+\frac{1}{3}(1-P)(1+\frac{1}{\rho})) \notag\\
			&+ (1-P)\big(\frac{2}{9} (1+\frac{1}{\rho})+1+\rho \big) \varepsilon_3= 3\rho^2 \varepsilon_3.\label{v3}
		\end{align}
	\end{subequations}
Therefore, by Lemma \ref{normA}, we obtain $d(C)\leq 3\rho^2$.
	\par By \eqref{conv_matrix2}, we obtain the following bound
		
		\begin{align*}
			\|u_{k+1}\|\leq {d}(C)^k
			\|u_1\|\!+\!\sum_{s=0}^{k-1} {d}(C)^s\bigg\|\begin{bmatrix}C_4 \\ 0 \\ C_4'' \end{bmatrix}\bigg\|\mathbb{E}\|\nabla f(\bar{x}^{k-s})\|^2,
		\end{align*}
		and consequently,
		
			\begin{align}\label{max_eq}
				\max&\Big\{\!\mathbb{E}\|{\bf y}^{k+1}\!-\!{\bf \bar{v}}^{k+1}\|^2,\!\mathbb{E}\|{\bf x}^{k+1}\!-\!{\bf \bar{x}}^{k+1}\|^2,\! \mathbb{E}\|\tau^{k+1}\!-\!{\bf \bar{x}}^{k+1}\|^2\!\Big\}\notag\\
				\leq& (3\rho^2)^k(\mathbb{E}\|{\bf y}^1-{\bf \bar{v}}^1\|^2+\mathbb{E}\|{\bf x}^1-{\bf \bar{x}}^1\|^2+\mathbb{E}\|\tau^1-{\bf \bar{x}}^1\|^2)\notag\\
				&+\sum_{s=0}^{k-1}(3\rho^2)^s(C_4+C_4'') \mathbb{E}\|\nabla f(\bar{x}^{k-s})\|^2,
			\end{align}
	where we used the fact that $\max\{a,b\}\leq \sqrt{a^2+b^2}\leq a+b$ for any $a\geq 0$ and $b\geq 0$.
\par	In addition, due to $3\rho^2<1$, we have
	\begin{align}\label{theta_s}
		\sum_{l=1}^{k}\sum_{s=0}^{l-1}(3\rho^2)^sa_{l-s}=&\sum_{l=1}^{k}\sum_{s=1}^{l}(3\rho^2)^{l-s}a_s\notag\\
		=&\sum_{s=1}^{k}a_s\sum_{l=s}^{k}(3\rho^2)^{l-s}\notag\\
		\leq &\frac{1}{1-3\rho^2}\sum_{s=1}^{k}a_s,
	\end{align}
 for any non-negative sequence $(a_s)_{s\in \mathbb{Z}^+}$.
	Then, by summing \eqref{max_eq} over iterations and \eqref{theta_s}, we obtain
	\begin{align}\label{sum_term}
		\max&\Big\{\sum_{s=1}^{k}\mathbb{E}\|{\bf y}^{s}-{\bf \bar{v}}^{s}\|^2,\sum_{s=1}^{k}\mathbb{E}\|{\bf x}^{s}-{\bf \bar{x}}^{s}\|^2, \sum_{s=1}^{k} \mathbb{E}\|\tau^{s}-{\bf \bar{x}}^{s}\|^2\Big\}\notag\\
		\leq & \frac{3\rho^2 R_0}{1-3\rho^2}+\frac{C_4+C_4''}{1-3\rho^2}\sum_{s=1}^{k}\mathbb{E}\|\nabla f(\bar{x}^{s})\|^2,
	\end{align}
	where $R_0=\mathbb{E}\|{\bf y}^1-{\bf \bar{v}}^1\|^2+\mathbb{E}\|{\bf x}^1-{\bf \bar{x}}^1\|^2+\mathbb{E}\|\tau^1-{\bf \bar{x}}^1\|^2$.
\end{proof}

In next proposition, we show the evolution of the cost function.
\begin{proposition}\label{f_propo}
	Suppose Assumption \ref{W_assump} hold. If the step-size satisfies $\eta\leq \frac{1}{6L}$,
	\begin{align}\label{nabla_c}
		& \mathbb{E} f(\bar{x}^{k+1})\notag\\
		\leq &\mathbb{E} f(\bar{x}^k)-\frac{\eta}{3}\mathbb{E} \|\nabla f(\bar{x}^k)\|^2\notag\\
		&+\frac{2L}{3n} \mathbb{E}\|{\bf x}^k-{\bf \bar{x}}^k\|^2+\frac{4L}{9n}\mathbb{E}\|\tau^k-{\bf \bar{x}} ^k\|^2 .
	\end{align}
\end{proposition}
\begin{proof}
	By $\bar{x}^{k+1}=\bar{x}^k-\eta \bar{v}^k$ and the $L$-smoothness of the function $f$, we have
	\begin{align*}
		f(\bar{x}^{k+1})\leq&  f(\bar{x}^k)-\eta  \langle\nabla f(\bar{x}^k), \bar{v}^k\rangle+\frac{\eta^2L}{2}\|\bar{v}^k\|^2\\
		=& f(\bar{x}^k)-\eta \|\nabla f(\bar{x}^k)\|^2+\frac{\eta^2L}{2}\|\bar{v}^k\|^2\\
		&-\eta \left<\nabla f(\bar{x}^k), \frac{1}{n}\sum_{i=1}^n (v_i^k-\nabla f_i(\bar{x}^k))\right>\\
		\leq &  f(\bar{x}^k)-\frac{\eta}{2} \|\nabla f(\bar{x}^k)\|^2+\frac{\eta^2L}{2} \|\bar{v}^k\|^2\\
		&+\frac{\eta}{2} \| \frac{1}{n}\sum_{i=1}^n (v_i^k-\nabla f_i(\bar{x}^k))\|^2
	\end{align*}
	Taking the conditional expectation with respect to $\mathcal{F}^{k+1}$, 
	\begin{align*}
		&\mathbb{E}[ f(\bar{x}^{k+1})|\mathcal{F}^{k+1}]\\
		\leq& f(\bar{x}^k)-\frac{\eta}{2} \|\nabla f(\bar{x}^k)\|^2+\frac{\eta^2L}{2}  \|\bar{v}^k\|^2\\
		&+\frac{\eta}{2} \| \frac{1}{n}\sum_{i=1}^n (v_i^k-\nabla f_i(\bar{x}^k))\|^2
	\end{align*} 
	Then, by \eqref{vsubfcondi} and \eqref{vcondi}, we have
	\begin{align}\label{fsubfstarcon}
		&\mathbb{E} [f(\bar{x}^{k+1})|\mathcal{F}^{k+1}]\notag\\
		\leq& f(\bar{x}^k)-\frac{\eta}{2} \|\nabla f(\bar{x}^k)\|^2\notag\\
		&+\eta^2L\Big( \|\nabla f(\bar{x}^k)\|^2+\frac{2L^2}{n}(3\|{\bf x}^k-{\bf \bar{x}}^k\|^2+2 \|{\bf \tau} ^k-{\bf \bar{x}}^k\|^2)\Big)\notag\\
		&+\frac{\eta}{2n}\big(6L^2 \|{\bf x}^k-{\bf \bar{x}}^k\|^2+4L^2 \|{\bf \tau} ^k-{\bf \bar{x}}^k\|^2\big)\notag\\
		=&f(\bar{x}^k)-(\!\frac{\eta}{2}-\eta^2L) \|\nabla f(\bar{x}^k)\|^2+\!\frac{3\eta L^2+6\eta^2L^3}{n} \!\|{\bf x}^k-{\bf \bar{x}}^k\|^2\notag\\
		&+\frac{4\eta^2L^3+2\eta L^2}{n} \|\tau^k-{\bf \bar{x}}^k\|^2.
	\end{align}
	Substitute $\eta L\leq \frac{1}{6}$ to the above inequality and take the total expectation on both sides of \eqref{fsubfstarcon}, we obtain \eqref{nabla_c}.
\end{proof}
\par  The following lemma will be used to establish the convergence of GT-VR.
\begin{lemma}\cite{ROBBINS1971233}\label{conv_l}
	Let ($\Omega,\mathcal{F},\mathcal{P}$) be a probability space and $(\mathcal{F}_t)_{t\in \mathbb{Z}^+}$ be a filtration. Let $U^t$, $\xi^t$ and $\zeta^t$ be nonnegative $\mathcal{F}_t$ measurable random variables for $t\in \mathbb{Z}^+$ such that 
	\begin{align}
		\mathbb{E}[U^{t+1}|\mathcal{F}_t]\leq U^t+\xi^t-\zeta^t, \quad \forall t=1,2,\cdots.
	\end{align}
	Then $U^t$ converges to a random variable and $\sum_{t=1}^{\infty} \zeta^t<+\infty$ almost surely on the event $\{\sum_{t=1}^{\infty} \xi^t<+\infty\}$.
\end{lemma}
 Now, we are ready to provide the proof of Theorem \ref{conver_theo}.
{\bf Proof of Theorem \ref{conver_theo}:}
\begin{proof}
	By proposition \ref{f_propo} and \eqref{sum_term}, we have
	\begin{align}\label{fsub}
		0\leq &f(\bar{x}^1)-f^*-\frac{\eta}{3}\sum_{s=1 }^{k}\mathbb{E}\|\nabla f(\bar{x}^{s})\|^2\notag\\
		&+(\frac{2L}{3n}\!+\!\frac{4L}{9n})(\frac{3\rho^2R_0}{1-3\rho^2}+\frac{C_4+C_4''}{1-3\rho^2}\sum_{s=1}^{k}\mathbb{E}\|\nabla f(\bar{x}^{s})\|^2)\notag\\
		=& f(\bar{x}^1)- f^*-\underbrace{(\frac{\eta}{3}-\frac{10L}{9 n}\frac{C_4+C_4''}{1-3\rho^2})}_{\mathcal{C}}\sum_{s=1}^{k}\mathbb{E}\|\nabla f(\bar{x}^{s})\|^2\notag\\
		&+\frac{10L}{9 n}\frac{3\rho^2R_0}{1-3\rho^2}\notag\\
		=&f(\bar{x}^1)-f^*-\mathcal{C}\sum_{s=1}^{k}\mathbb{E}\|\nabla f(\bar{x}^{s})\|^2+\frac{10L}{9n }\frac{3\rho^2R_0}{1-3\rho^2}.
	\end{align}
	Since
	\begin{align}\label{C_pos}
		0<\eta < \frac{(1-3\rho^2)}{\big(16\rho^2 L^2+(32\rho^2 L^2+2)(1-P)\frac{\rho+1}{\rho}\big)5L},
	\end{align}
	then,
	\begin{align*}
		\mathcal{C}=&\frac{\eta}{3}\!-\!\frac{10L}{9}\frac{1}{1\!-\!3\rho^2}\Big(16 \rho^2 \eta^2 L^2+32 (1\!-\! P) \rho^2 \eta^2 L^2(1+\frac{1}{\rho})\\
		&+2 \eta^2  (1-P)(1+\frac{1}{\rho})\Big)\\
		=& \frac{\eta}{3}\bigg(1\!-\!\frac{10L}{3}\frac{1}{1\!-\! 3\rho^2}\Big( 16 \rho^2 \eta L^2+32 (1\!-\! P) \rho^2 \eta L^2(1+\frac{1}{\rho})\\
		&+2 \eta  (1-P)(1+\frac{1}{\rho})\Big)\bigg)\\
		>&0.
	\end{align*}
	\par It follows from \eqref{fsub} that
	\begin{align}\label{nabla_conv}
		\frac{1}{k}\sum_{s=1}^{k}\mathbb{E}\|\nabla f(\bar{x}^{s})\|^2\leq \frac{1}{k\mathcal{C}}\big[ \big(f(\bar{x}^1)-f^*\big)+\frac{10L}{9 n}\frac{3\rho^2R_0}{1-3\rho^2}\big],
	\end{align} 
which implies that $\frac{1}{k}\sum_{s=1}^{k}\mathbb{E}\|\nabla f(\bar{x}^{s})\|^2\leq \mathcal{O}(\frac{1}{k})$ holds in Theorem \ref{conver_theo}.
\par Now, by  \eqref{sum_term} and \eqref{nabla_conv}, we see that
	\begin{align*}
		&\max\Big\{\frac{1}{k}\sum_{s=1}^{k}\mathbb{E}\|{\bf x}^{s}-{\bf \bar{x}}^{s}\|^2,\frac{1}{k}\sum_{s=1}^{k}\|\tau^{s}-{\bf \bar{x}}^{s}\|^2 \Big\}\notag\\
		\leq& \frac{1}{k}\!\bigg(\frac{R_0}{1\!-\!3\rho^2}+\!\frac{C_4+C_4''}{1\!-\!3\rho^2}\Big(\frac{1}{\mathcal{C}} \big(f(\bar{x}^1)-f^*\big)+\!\frac{10L}{9n\mathcal{C}}\frac{3\rho^2R_0}{1\!-\!3\rho^2}\Big)\bigg)\notag\\
		=&\frac{1}{k}\!\bigg(\frac{C_4+C_4''}{(1\!-\!3\rho^2)\mathcal{C}} \big(f(\bar{x}^1)\!-\!f^*\big)\!+\!\Big(\frac{C_4\!+\!C_4''}{1\!-\!3\rho^2}\frac{10L}{9n\mathcal{C}}+1\Big)\frac{3\rho^2R_0}{1\!-\!3\rho^2}\bigg),
	\end{align*}
which implies that $\frac{1}{k}\sum_{s=1}^{k}\mathbb{E}\|{\bf x}^{s}-{\bf \bar{x}}^{s}\|^2\leq \mathcal{O}(\frac{1}{k})$ in Theorem \ref{conver_theo} holds.
	\par In addition, by \eqref{vsubfcondi}, we have
		\begin{align*}
			&\frac{1}{k}\sum_{s=1}^{k}\mathbb{E}\|{\bf y}^{s}-\nabla f(\mathbf{\bar{x}}^{s})\|^2\notag\\
			\leq & \frac{2}{k}\sum_{s=1}^{k}(\mathbb{E}\|{\bf y}^{s}-{\bf \bar{v}}^{s}\|^2+\mathbb{E}\|{\bf \bar{v}}^{s}- \nabla f(\mathbf{\bar{x}}^{s})\|^2)\\
			\leq & \frac{2}{k}\sum_{s=1}^{k}(\mathbb{E}\|{\bf y}^{s}-{\bf \bar{v}}^{s}\|^2\!+\!{6L^2}\mathbb{E}\|{\bf x}^{s}\!-\!{\bf \bar{x}}^{s}\|^2\!+\!{4L^2}\mathbb{E}\|{\bf \tau}^{s}-{\bf \bar{x}}^{s}\|^2)\\
			\leq& \frac{1}{k}(2+{20L^2})\Big(\frac{C_4+C_4''}{(1-3\rho^2)\mathcal{C}} \big(f(\bar{x}^1)-f^*\big)\\
			&\quad+\Big(\frac{C_4+C_4''}{1-3\rho^2}\frac{10L}{9n\mathcal{C}}+1\Big)\frac{3\rho^2R_0}{1-3\rho^2}\Big),
	\end{align*}
which implies that $\frac{1}{k}\sum_{s=1}^{k}\mathbb{E}\|{\bf y}^{s}-\nabla f(\mathbf{\bar{x}}^{s})\|^2\leq \mathcal{O}(\frac{1}{k})$ in Theorem \ref{conver_theo} holds.
	Finally, by Proposition \ref{f_propo} and Lemma \ref{conv_l}, we see that $f(\bar{x}^k)$ converges and $\sum_{k=1}^{\infty} E[\|\nabla f(\bar{x}^k)\|^2]<+\infty$.
\end{proof}
\par With the convergence result in Theorem \ref{conver_theo} and the definition of $\epsilon$-accurate stationary point in Definition \ref{acc_def}, we present the analysis of Corollary \ref{conver_corol} in the following.
\par {\bf Proof of Corollary \ref{conver_corol}:}
\begin{proof}
	By definition \ref{acc_def} and \eqref{nabla_conv}, it is clear that to find an $\epsilon$-accurate stationary point, it is sufficient to find the iterations $k$ such that
	\begin{align}\label{eq1}
		&\frac{1}{k\mathcal{C}}\big[ \big(f(\bar{x}^1)-f^*\big)+\frac{10L}{9 n}\frac{3\rho^2R_0}{1-3\rho^2}\big]\leq \epsilon
	\end{align}
where $\mathcal{C}=\frac{\eta}{3}\Big(1-\frac{10L}{3}\frac{1}{1-3\rho^2}\big( 16 \rho^2 \eta L^2+32 (1-P) \rho^2 \eta L^2(1+\frac{1}{\rho})+2 \eta  (1-P)(1+\frac{1}{\rho})\big)\Big)$.
	If the step-size $\eta$ satisfies \eqref{eta_acc}, $\mathcal{C}\geq \frac{\eta}{9}$ and $\frac{1}{\eta L}\geq \frac{3\rho^2}{1-3\rho^2}$ hold. Thus, when the iteration $k$ satisfies
	\begin{align}
		\frac{9}{\eta \epsilon}\big[ \big(f(\bar{x}^1)-f^*\big)+\frac{10}{9 n}\frac{R_0}{\eta }\big]\leq k,
	\end{align}
then, the inequality \eqref{eq1} holds.
	Therefore, the iteration complexity of GT-VR is $\mathcal{O}(\frac{1}{\eta \epsilon}\big[ \big(f(\bar{x}^1)-f^*\big)+\frac{R_0}{n \eta }\big])$. 
	\par For the gradient computations, since there are $Pm_i+2$ gradient computations at each iteration of agent $i$, the number of gradient computations across all agents is the iteration complexity multiplied by $\sum_{i=1}^n (P m_i+2)$. Since $\sum_{i=1}^n m_i=M$, the computational complexity is $\mathcal{O}(\frac{PM+n}{\eta \epsilon}\big[ \big(f(\bar{x}^1)-f^*\big)+\frac{R_0}{n \eta }\big])$. 
	\par At each iteration, each agent $i$ communicate twice with its neighbors. Let $N_i$ denote the number of neighbors of agent $i$. Then, the communication complexity across all agents is $\mathcal{O}(\frac{\sum_{i=1}^n N_i}{\eta \epsilon}\big[ \big(f(\bar{x}^1)-f^*\big)+\frac{R_0}{n \eta }\big])$ by multiplying the iteration complexity by $\sum_{i=1}^n N_i$. 
\end{proof}

\section{Simulation}\label{simulation}
To verify the efficacy of the proposed algorithm, we consider the classical binary classification problem, which is to find one optimal predictor $x\in \mathbb{R}^d$ on a popular logistic regression learning model. We compare the proposed algorithm with recently proposed algorithms GT-SAGA\cite{xin2020fast}, GT-SARAH\cite{xin2020nearoptimal} and D-GET\cite{sun2019improving}. The learning model is to optimize the following problem
\begin{align}\label{simu_pro}
	\min_{x\in \mathbb{R}^d} f(x)&\triangleq\frac{1}{n} \sum_{i=1}^n f_i(x),\notag\\ 
	f_i(x)&=\frac{1}{m_i} \sum_{j=1}^{m_i} \frac{1}{1+\exp(l_{ij}a_{ij}'x)}+\lambda_1\|x\|_2^2,
\end{align}
where $a_{ij}\in \mathbb{R}^d$, $l_{ij}\in \{-1,1\}$ and $\{a_{ij},l_{ij}\}_{j=1}^{m_i}$ denotes the set of training samples of agent $i$. 
\begin{table}[!htbp]
	\caption{Real data for black-box binary classification}\label{real_data}
	\centering
	\begin{tabular}{c|c|c|c}
		\hline
		datasets & \#samples & \#features & \#classes  \\
		\hline
		$a9a$  & 32561& 123 &2\\
		\hline
		w8a & 64700&300&2\\
		\hline
		$covtype.binary$&581012&54&2\\
		\hline
	\end{tabular}
\end{table}
\begin{figure}
	\centering
	\subfigure[a9a dataset]{
		\includegraphics[width=8 cm, height = 4.7 cm]{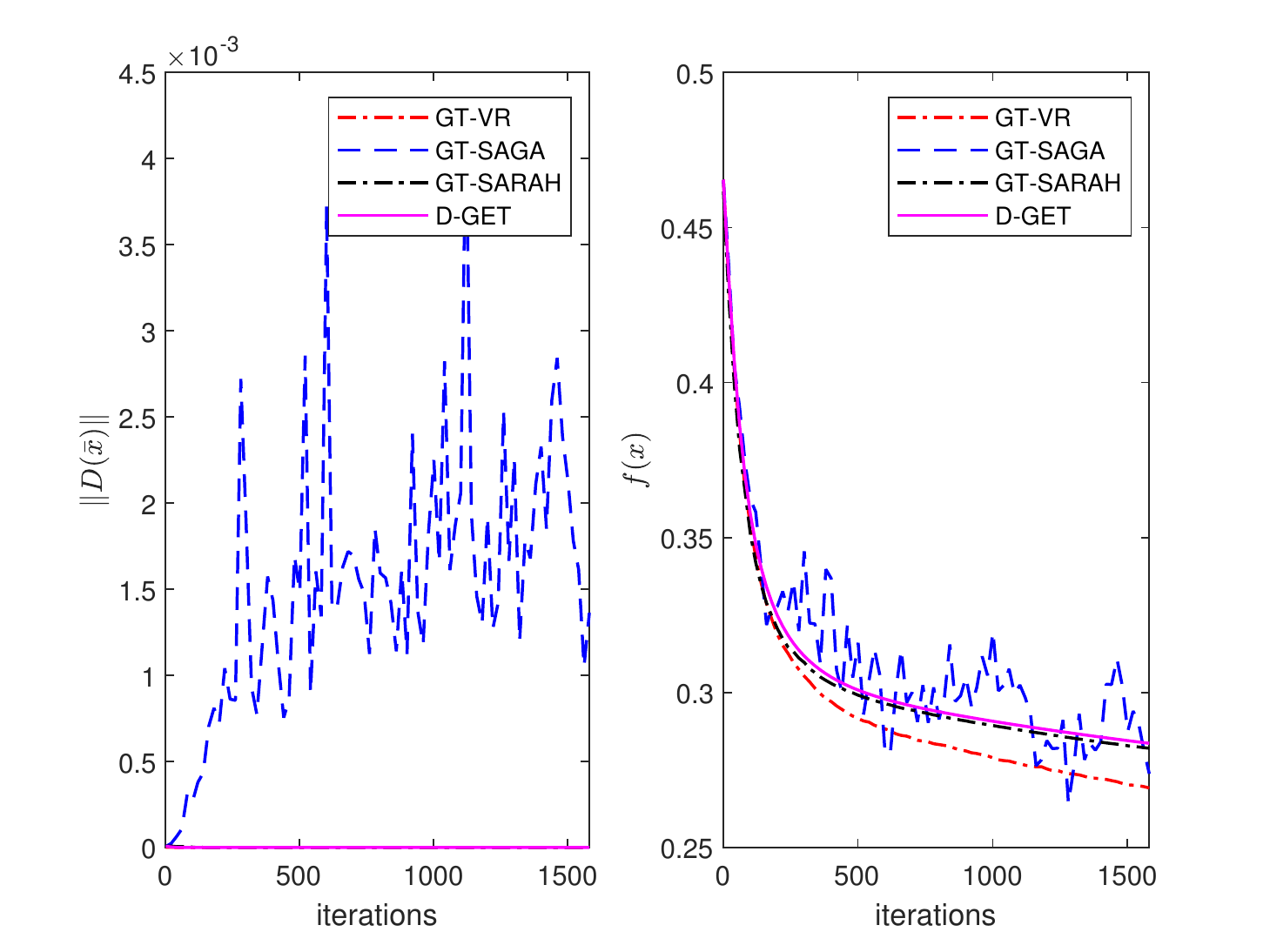}
		\label{a9a_fig}
	}
	
	\subfigure[w8a dataset]{
		\includegraphics[width=8 cm, height = 4.7 cm]{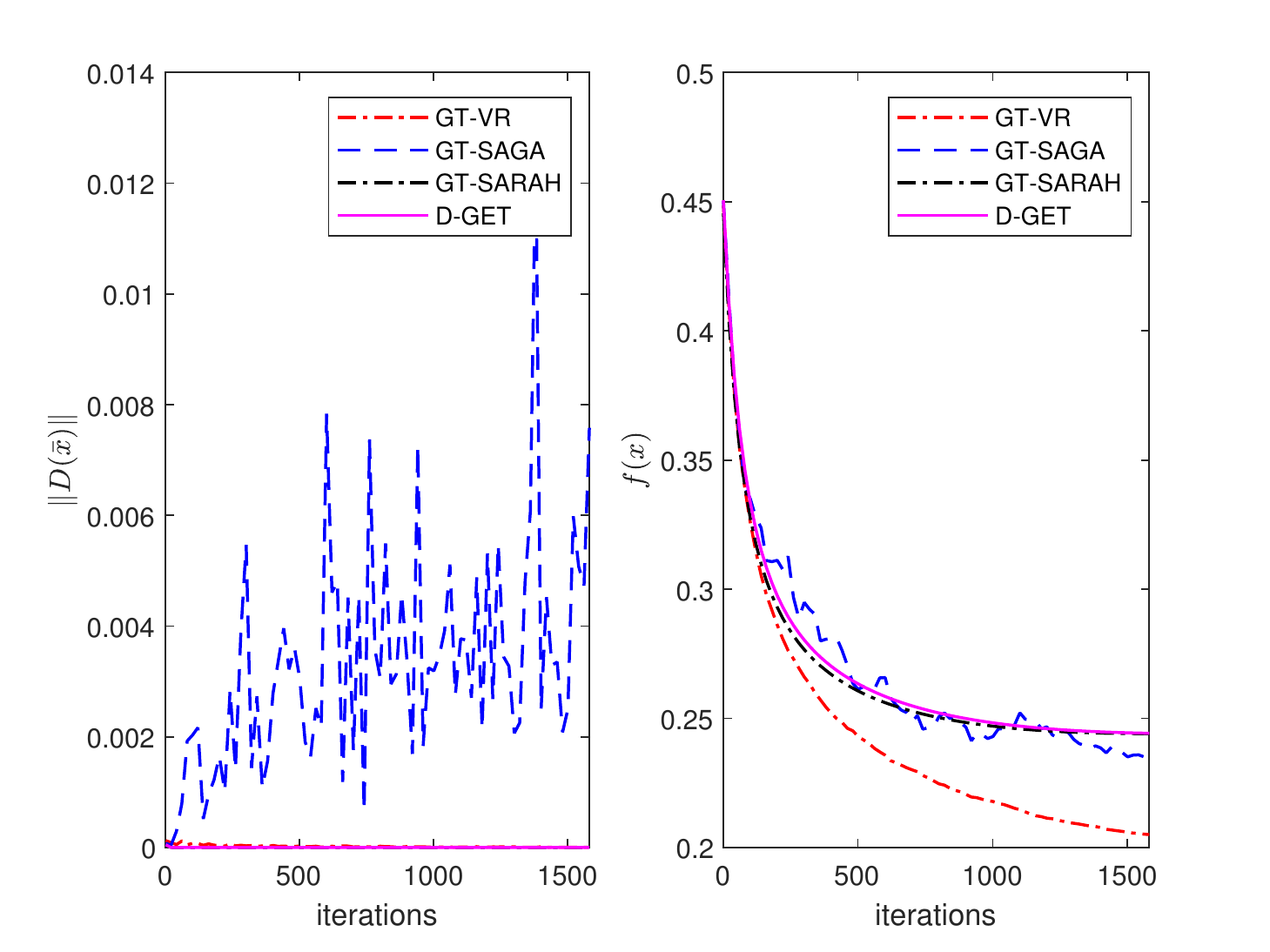}
		\label{w8a_fig}
	}
	\subfigure[covtype.binary dataset]{
		\includegraphics[width=8 cm, height = 4.7 cm]{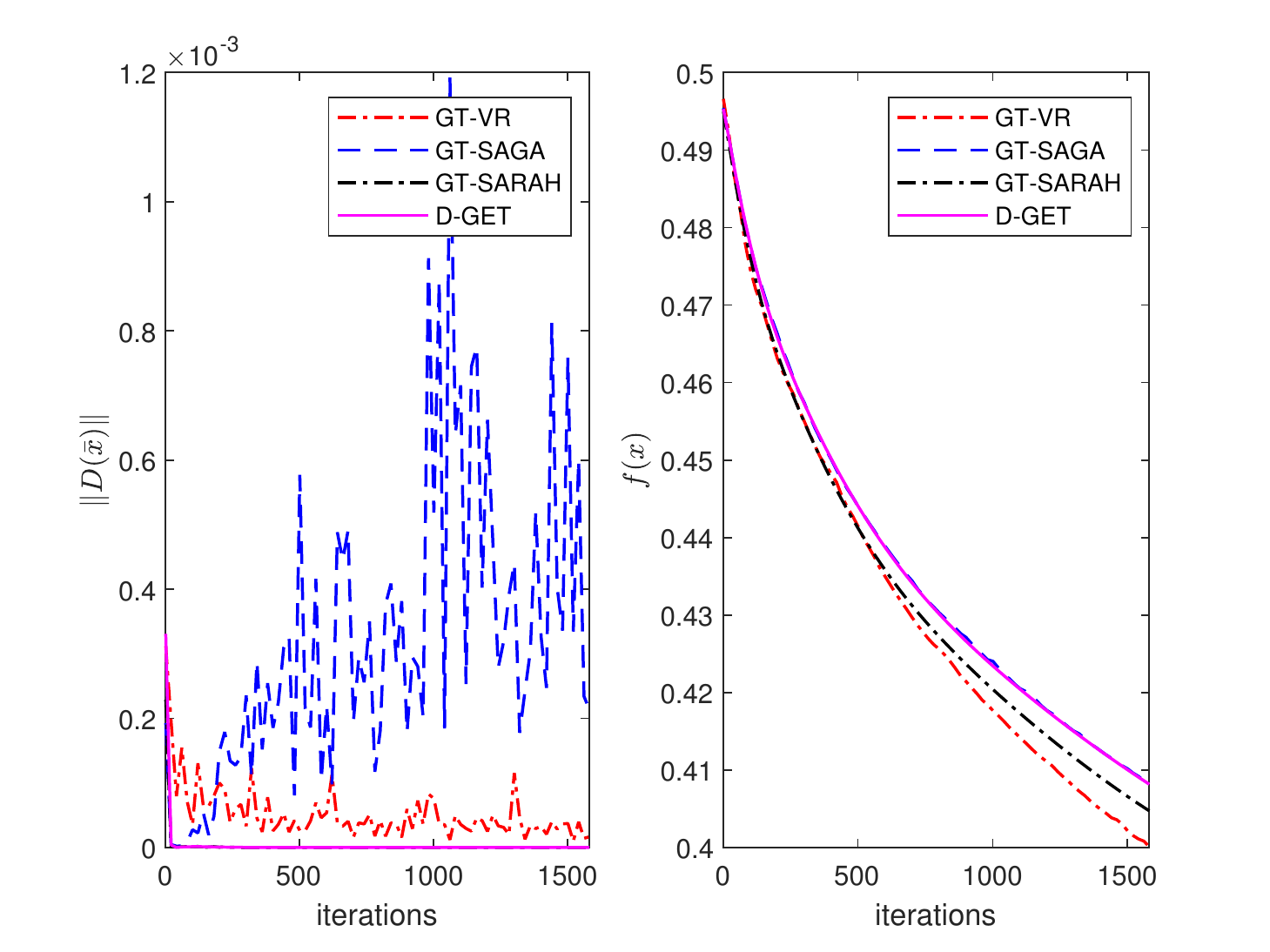}
		\label{cov_fig}
	}
	\caption{The convergence behaviors of GT-VR, GT-SAGA, GT-HSGD, GT-SARAH, D-GET over the a9a and w8a datasets}
	\label{conv_fig}
\end{figure}

In this experiment, we use the publicly available real datasets\footnote{$a9a
	$, $w8a$ and $covtype.binary$ are from the website
	www.csie.ntu.edu.tw/~cjlin/libsvmtools/datasets/.}, which are summarized in Table \ref{real_data}. All algorithms are applied over a ten-agent undirected connected network with a doubly stochastic adjacent matrix $W$ to solve \eqref{simu_pro}. Meanwhile, $\lambda_1\triangleq 5\times 10^{-4}$. For the proposed GT-VR algorithm, the probability $P$ is taken as $0.3$ and the step-size is taken as $0.1$. Note that the ranges of step-size and possibility provided in Theorem \ref{conver_theo} are rigorous theoretical results. In practice, we can adjust them according to the convergence performance. For comparison, the local variables in all algorithms are initialized as zero and all the algorithms take same step-sizes. The simulation codes are provided at \url{https://github.com/managerjiang/GT_VR_Simulation}.
\par Define $D(\bar{x})=\sum_{i=1}^{10}x_i'\sum_{j=1}^{10} w_{ij}(x_i-x_j)$. The trajectory of $D(\bar{x})$ converging to zero implies that the variable estimates of different agents achieve consensus. For different datasets, the trajectories of cost function $f$ and $D(\bar{x})$ are shown in Fig. \ref{conv_fig}. We observe that for all datasets, the algorithms all have good consensus performance. For the trajectories of cost function, the algorithm GT-VR decays faster than the state-of-the-art algorithms D-GET, GT-SAGA and GT-SARAH, especially for the w8a dataset, demonstrating the excellent iteration complexity of GT-VR.
%

\section{Conclusion}\label{conclusion} 
Focusing on distributed non-convex stochastic optimization, this paper has developed a novel variance-reduced distributed stochastic first-order  algorithm over undirected and weight-balanced directed graphs by combining gradient tracking and variance reduction. The variance reduction technique makes use of Bernoulli distribution to handle the variance by stochastic gradients. The proposed algorithm converges with $O(\frac{1}{k})$ rate and has lower iteration complexity compared with some existing excellent algorithms. By comparative simulations, the proposed algorithm presents better convergence performance than state-of-the-art algorithms.

\section{appendix}\label{append}
 \subsection{\bf Proof of Proposition \ref{vnablaf}}\label{vnablaproof}
\begin{proof}
	(a) For convenience, we define $\mathcal{A}^k\triangleq \sigma(\cup_{i=1}^n \sigma(l_i^k),\mathcal{F}^k)$ and clearly $\mathcal{F}^k \subseteq \mathcal{A}^k$. By the tower property of the conditional expectation, we have 
	\begin{align}\label{tower_con0}
		\mathbb{E}[\|v_i^k-\nabla f_i(x_i^k)\|^2|\mathcal{F}^k]=\mathbb{E}[\mathbb{E}[\|v_i^k-\nabla f_i(x_i^k)\|^2|\mathcal{A}^k]|\mathcal{F}^k].
	\end{align}
	For $\mathbb{E}[\|v_i^k-\nabla f_i(x_i^k)\|^2|\mathcal{A}^k]$, we have
	\begin{align}\label{vsubf}
		&\mathbb{E}[\|v_i^k-\nabla f_i(x_i^k)\|^2|\mathcal{A}^k]\notag\\
		=&\mathbb{E}[\|\nabla f_{i,s_i^k}(x_i^k)\!-\!\nabla f_{i,s_i^k}(\tau_i^k)\!-\!(\nabla f_i(x_i^k)\!-\!\nabla f_i(\tau_i^k))\|^2|\mathcal{A}^k]\notag\\
		\leq &\mathbb{E}[\|\nabla f_{i,s_i^k}(x_i^k)-\nabla f_{i,s_i^k}(\tau_i^k)\|^2|\mathcal{A}^k]\notag\\
		=& \frac{1}{m_i}\sum_{j=1}^{m_i}\|\nabla f_{i,j}(x_i^k)\!-\!\nabla f_{i,j}(\bar{x}^k)+\nabla f_{i,j}(\bar{x}^k)\!-\!\nabla f_{i,j}(\tau_i^k)\|^2\notag\\
		\leq & 2 L^2 \|x_i^k-\bar{x}^k\|^2+2L^2\|\tau_i^k-\bar{x}^k\|^2
	\end{align}
	where the first inequality is from the standard conditional variance decomposition
	\begin{align*}
		&\mathbb{E}[\|a_i^k-\mathbb{E}[a_i^k|\mathcal{A}^k]\|^2|\mathcal{A}^k]\\
		=&\mathbb{E}[\|a_i^k\|^2|\mathcal{A}^k]-\|\mathbb{E}[a_i^k|\mathcal{A}^k]\|^2\\
		\leq & \mathbb{E}[\|a_i^k\|^2|\mathcal{A}^k],
	\end{align*}
	with $a_i^k=\nabla f_{i,s_i^k}(x_i^k)-\nabla f_{i,s_i^k}(\tau_i^k)$. 
	Substitute the above inequality to \eqref{tower_con0},
	\begin{align}\label{vsubnablaf}
		&\mathbb{E}[\|v_i^k-\nabla f_i(x_i^k)\|^2|\mathcal{F}^k]\notag\\
		\leq &2 L^2 \mathbb{E}[\|x_i^k-\bar{x}^k\|^2|\mathcal{F}^k]+2L^2\mathbb{E}[\|\tau_i^k-\bar{x}^k\|^2\mathcal{F}^k].
	\end{align}
	The proof follows by summing \eqref{vsubnablaf} over $i$ and taking the total expectation on both sides.
	\par (b) With the result in (a), $\mathbb{E} [\|\frac{1}{n}\sum_{i=1}^n(v_i^k-\nabla f_i(\bar{x}^k))\|^2|\mathcal{F}^k]$ satisfies \begin{align}\label{vinablfi}
		&\mathbb{E} [\|\frac{1}{n}\sum_{i=1}^n(v_i^k-\nabla f_i(\bar{x}^k))\|^2|\mathcal{F}^k]\notag\\
		\leq & \frac{1}{n}\sum_{i=1}^n \mathbb{E}[\|v_i^k-\nabla f_i(\bar{x}^k)\|^2|\mathcal{F}^k]\notag\\
		\leq &\frac{1}{n}\sum_{i=1}^n(2\mathbb{E}[\|v_i^k-\nabla f_i(x_i^k)\|^2|\mathcal{F}^k]\notag\\
		&+2\mathbb{E}[\|\nabla f_i(x_i^k)-\nabla f_i(\bar{x}^k)\|^2|\mathcal{F}^k])\notag\\
		\overset{ \eqref{vsubnablaf}}{\leq} &\frac{1}{n}\sum_{i=1}^n\big(2(2 L^2 \mathbb{E}[\|x_i^k \!-\!\bar{x}^k\|^2|\mathcal{F}^k]+2L^2\mathbb{E}[\|\tau_i^k \!-\!\bar{x}^k\|^2|\mathcal{F}^k])\notag\\
		&+2L^2\mathbb{E}[\|x_i^k-\bar{x}^k\|^2|\mathcal{F}^k]\big)\notag\\
		=&\frac{1}{n}\sum_{i=1}^n(6L^2\mathbb{E}[\|x_i^k-\bar{x}^k\|^2|\mathcal{F}^k]+4L^2\mathbb{E}[\|\tau_i^k-\bar{x}^k\|^2|\mathcal{F}^k]).
	\end{align}
	The proof follows by taking the total expectation on both sides of \eqref{vinablfi}.
	\par (c) By Young's inequality and the result in (b), $\mathbb{E}[\|\bar{v}^k\|^2]$ satisfies \begin{align*}
		\mathbb{E}[\|\bar{v}^k\|^2]\leq &2\mathbb{E}[\|\nabla f(\bar{x}^k)\|^2]+2\mathbb{E}[\|\frac{1}{n}\sum_{i=1}^n(v_i^k-\nabla f_i(\bar{x}^k))\|^2]\\
		\leq & 2\mathbb{E}[\|\nabla f(\bar{x}^k)\|^2]+\frac{2}{n}\sum_{i=1}^n(6L^2\mathbb{E}[\|x_i^k-\bar{x}^k\|^2]\\
		&+4L^2\mathbb{E}[\|\tau_i^k-\bar{x}^k\|^2]),
	\end{align*}
where the last inequality follows from \eqref{vsubfcondi}.
\end{proof}
\subsection {\bf Proof of Proposition \ref{xsub}} \label{x_proof}
 \begin{proof} 
 \par (a) Recall that $\bar{x}^k=\mathbf{1}_n\otimes \bar{x}^k$.  By \eqref{x_up} and \eqref{bareq},
 \begin{align}\label{xbarx}
 &\|{\bf x}^{k+1}-{\bf \bar{x}}^{k+1}\|^2\notag\\
 =&\|(W\otimes I_d) [{\bf x}^k-{\bf \bar{x}}^k-\eta({\bf y}^k-{\bf \bar{v}}^k)]\|^2\notag\\
 \leq & 2\rho^2\|{\bf x}^k-{\bf \bar{x}}^k\|^2+2\rho^2\eta^2\|{\bf y}^k-{\bf \bar{v}}^k\|^2.
 \end{align}
We take the total expectation on both sides of \eqref{xbarx} to obtain
 \begin{align}\label{x_1}
 &\mathbb{E}[\|{\bf x}^{k+1}-{\bf \bar{x}}^{k+1}\|^2]\notag\\
 \leq &2\rho^2\mathbb{E}[\|{\bf x}^k-{\bf \bar{x}}^k\|^2]+2\rho^2\eta^2\mathbb{E}[\|{\bf y}^k-{\bf \bar{v}}^k\|^2].
 \end{align}

(b) Recall the Bernoulli distribution in GT-VR that 
$$\mathbb{E}[\mathbb{I}_{\{l_i^{k+1}= 1\}}|\mathcal{F}^{k+1}]=P \ {\rm and} \ \mathbb{E}[\mathbb{I}_{\{l_i^{k+1}\neq 1\}}|\mathcal{F}^{k+1}]=1-P.$$
Then,  
\begin{align}\label{xsubtaucondi2}
&\mathbb{E}[\|\bar{x}^{k+1}-\tau_i^{k+1}\|^2|\mathcal{F}^{k+1}]\notag\\
=& \mathbb{E}\big[\|\bar{x}^{k+1}-\big(\mathbb{I}_{\{l_i^{k+1}=1\}} x_i^{k+1}+\mathbb{I}_{\{l_i^{k+1}\neq 1\}} \tau_i^{k}\big)\|^2|\mathcal{F}^{k+1}\big]\notag\\
=& \mathbb{E}[\|\bar{x}^{k+1}\|^2|\mathcal{F}^{k+1}]\notag\\
&+ \mathbb{E}[\|\mathbb{I}_{\{l_i^{k+1}= 1\}} x_i^{k+1} +\mathbb{I}_{\{l_i^{k+1}\neq 1\}} \tau_i^k\|^2|\mathcal{F}^{k+1}]\notag\\
&-2\mathbb{E}\big[\langle \bar{x}^{k+1}, \mathbb{I}_{\{l_i^{k+1}= 1\}} x_i^{k+1}+\mathbb{I}_{\{l_i^{k+1}\neq 1\}} \tau_i^k\rangle|\mathcal{F}^{k+1} \big]\notag\\
=& \|\bar{x}^{k+1}\|^2+P \| x_i^{k+1}\|^2 + (1-P)\| \tau_i^k\|^2\notag\\
&-2\langle \bar{x}^{k+1}, P x_i^{k+1}+(1-P) \tau_i^k\rangle\notag\\
=& P \|\bar{x}^{k+1}-x_i^{k+1}\|^2 + (1-P)\|\bar{x}^{k+1}-\tau_i^k\|^2.
\end{align}
Summing over $i$ and taking the total expectation on both sides,
\begin{align*}
&\mathbb{E}[\|\mathbf{\bar{x}}^{k+1}-\tau^{k+1}\|^2]\\
=&P\mathbb{E}[\mathbf{\|\bar{x}}^{k+1}-\mathbf{x}^{k+1}\|^2]+(1-P) \mathbb{E}[\|\mathbf{\bar{x}}^{k+1}-\tau^k\|^2].
\end{align*}
For the second term,
\begin{align}\label{t2}
&\mathbb{E}[\|\mathbf{\bar{x}}^{k+1}-\tau^k\|^2] \notag\\
=&\mathbb{E}[ \|\mathbf{\bar{x}}^{k+1}-\mathbf{\bar{x}}^k +\mathbf{\bar{x}}^k -\tau^k\|^2]\notag\\
=&\mathbb{E}[\eta^2 n \|\bar{v}^k\|^2 +\|\mathbf{\bar{x}}^k -\tau^k\|^2-2\eta \langle \mathbf{\bar{v}}^k,\mathbf{\bar{x}}^k -\tau^k \rangle] \notag\\
\leq & \mathbb{E}[\eta^2 n \|\bar{v}^k\|^2 +\|\mathbf{\bar{x}}^k -\tau^k\|^2+\frac{\eta}{\beta} n \|\bar{v}^k\|^2+\eta\beta \|\mathbf{\bar{x}}^k -\tau^k\|^2] \notag\\
=& (\eta^2+\frac{\eta}{\beta})n \mathbb{E}[\|\bar{v}^k \|^2] + (1+\eta\beta  )\mathbb{E}[\|\mathbf{\bar{x}}^k -\tau^k\|^2] \notag\\
\leq & (\eta^2+\frac{\eta}{\beta}) n \big( 2 \mathbb{E}\|\nabla f(\bar{x}^k)\|^2+\frac{12L^2}{n} \mathbb{E}\|\mathbf{x}^k-\mathbf{\bar{x}}^k\|^2\notag\\
&+\frac{8L^2}{n} \mathbb{E}\|\tau^k-\mathbf{\bar{x}}^k\|^2\big)+ (1+\eta\beta  )\mathbb{E}[\|\mathbf{\bar{x}}^k -\tau^k\|^2],
\end{align}
where the last inequality holds by Proposition \ref{vnablaf} (c).
\par By \eqref{x_1} and \eqref{t2},
\begin{align*}
&\mathbb{E}[\|\mathbf{\bar{x}}^{k+1}-\tau^{k+1}\|^2]\\
\leq &P \big(2\rho^2 \mathbb{E}\|\mathbf{\bar{x}}^k-\mathbf{x}^k\|^2+ 2\rho^2 \eta^2 \mathbb{E}\|\mathbf{y}^k-\mathbf{\bar{v}}^k\|^2 \big)\\
&+(1\!-\! P) \Big((\eta^2\!+\!\frac{\eta}{\beta}) \big( 2n \mathbb{E}\|\nabla f(\bar{x}^k)\|^2+12L^2\mathbb{E}\|\mathbf{x}^k-\mathbf{\bar{x}}^k\|^2\\
&+8L^2\mathbb{E}\|\tau^k-\mathbf{\bar{x}}^k\|^2\big)+ (1+\eta\beta)\mathbb{E}\|\mathbf{\bar{x}}^k -\tau^k\|^2\Big)\\
=& \big(2\rho^2 P +(1-P)(\eta^2+\frac{\eta}{\beta})12L^2\big) \mathbb{E}\|\mathbf{\bar{x}}^k-\mathbf{x}^k\|^2 \\
&+ 2\rho^2 \eta^2 P \mathbb{E}\|\mathbf{y}^k-\mathbf{\bar{v}}^k\|^2+(1-P)\big((\eta^2+\frac{\eta}{\beta})8L^2\\
&+(1\!+\!\eta\beta)\big) \mathbb{E}\|\mathbf{\bar{x}}^k \!-\!\tau^k\|^2\!+\! (\eta^2+\frac{\eta}{\beta}) 2n(1\!-\! P) \mathbb{E}\|\nabla f(\bar{x}^k)\|^2.
\end{align*}

(c) By \eqref{y_up}, we have 
	\begin{align}\label{ysubv}
	&\mathbf{y}^{k+1}-\mathbf{\bar{v}}^{k+1}\notag\\
	=&(W\otimes I_d)(\mathbf{y}^k-\mathbf{\bar{v}}^k+\mathbf{v}^{k+1}-\mathbf{v}^k-\mathbf{\bar{v}}^{k+1}+\mathbf{\bar{v}}^k).
	\end{align}
	\par For the term $\mathbf{v}^{k+1}-\mathbf{v}^k-\mathbf{\bar{v}}^{k+1}+\mathbf{\bar{v}}^k$ in \eqref{ysubv}, it satisfies 
	\begin{align}\label{vvbarv}
	&\|\mathbf{v}^{k+1}-\mathbf{v}^k-\mathbf{\bar{v}}^{k+1}+\mathbf{\bar{v}}^k\|^2\notag\\
	=&\|\mathbf{v}^{k+1}\!-\!\mathbf{v}^k\|^2\!+\!\|\mathbf{\bar{v}}^{k+1}\!-\!\mathbf{\bar{v}}^k\|^2\!-\!2\sum_{i=1}^n\langle{v}_i^{k+1}\!-\!{v}_i^k,{\bar{v}}^{k+1}\!-\!{\bar{v}}^k\rangle\notag\\
	=&\|\mathbf{v}^{k+1}-\mathbf{v}^k\|^2-\|\mathbf{\bar{v}}^{k+1}-\mathbf{\bar{v}}^k\|^2\notag\\
	\leq & \|\mathbf{v}^{k+1}-\mathbf{v}^k\|^2.
	\end{align}
	Hence, it follows from Lemma \ref{W_lemma}, \eqref{ysubv} and \eqref{vvbarv} that
	\begin{align*}
	&\|\mathbf{y}^{k+1}-\mathbf{\bar{v}}^{k+1}\|^2\\
	\leq & 2\rho^2 (\|\mathbf{y}^k-\mathbf{\bar{v}}^k\|^2+\|\mathbf{v}^{k+1}-\mathbf{v}^k\|^2).
	\end{align*}
	 By taking the total expectation on both sides,
	\begin{align*}
	\mathbb{E}\|{\bf y}^{k+1}\!-\!{\bf \bar{v}}^{k+1}\|^2\leq 2\rho^2 \mathbb{E}\|{\bf y}^k \! -\!{\bf \bar{v}}^k\|^2\!+\! 2 \rho^2 \mathbb{E}\|\mathbf{v}^{k+1}\!-\!\mathbf{v}^k\|^2.
	\end{align*}
	The second term $\mathbb{E}\|\mathbf{v}^{k+1}\!-\!\mathbf{v}^k\|^2$ in the above inequality satisfies
	\begin{align}\label{vv1}
	&\mathbb{E}\|\mathbf{v}^{k+1}-\mathbf{v}^k\|^2\notag\\
	\leq &2\mathbb{E}[\|\nabla f(\mathbf{x}^{k+1})-\nabla f(\mathbf{x}^k)\|^2]\notag\\
	&+2\mathbb{E}[\|\mathbf{v}^{k+1}-\mathbf{v}^k-\nabla f(\mathbf{x}^{k+1})+\nabla f(\mathbf{x}^k)\|^2]\notag\\
	\leq & 2L^2\mathbb{E}\|\mathbf{x}^{k+1}\!-\!\mathbf{x}^k\|^2\!+\!2\mathbb{E}[\|\mathbf{v}^{k+1}\!-\!\mathbf{v}^k \!-\!\nabla f(\mathbf{x}^{k+1})\!+\!\nabla f(\mathbf{x}^k)\|^2]\notag\\
	\leq & 2L^2\mathbb{E}\|\mathbf{x}^{k+1}-\mathbf{x}^k\|^2+4\mathbb{E}\|\mathbf{v}^{k+1}-\nabla f(\mathbf{x}^{k+1})\|^2\notag\\
	&+4\mathbb{E}\|\mathbf{v}^k-\nabla f(\mathbf{x}^k)\|^2\notag\\
	\leq &  2L^2\mathbb{E}\|\mathbf{x}^{k+1}-\mathbf{x}^k\|^2+8L^2 \mathbb{E}\|\mathbf{x}^{k+1}-\mathbf{\bar{x}}^{k+1}\|^2\notag\\
	&+8L^2\mathbb{E}\|\tau^{k+1}-\mathbf{\bar{x}}^{k+1}\|^2\notag\\
	&+8L^2 \mathbb{E}\|\mathbf{x}^k-\mathbf{\bar{x}}^k\|^2+8L^2\mathbb{E}\|\tau^k-\mathbf{\bar{x}}^k\|^2,
	\end{align}
	where the last equality is from the result in Proposition \ref{vnablaf} (a). Then, for the first term $\|{\bf x}^{k+1}-{\bf x}^k\|^2$ in \eqref{vv1}, it satisfies
	\begin{align*}
	&\|{\bf x}^{k+1}-{\bf x}^k\|^2\\
	=&\|(W\otimes I_d -I_d){\bf x}^k-\eta (W\otimes I_d){\bf y}^k\|^2\\
	=&\|(W\otimes I_d -I_d)\big({\bf x}^k-\mathbf{\bar{x}}^k\big)-\eta (W\otimes I_d)\big({\bf y}^k-\mathbf{\bar{v}}^k\big)\\
	&-\eta \mathbf{1}_n \otimes { \bar{v}}^k\|^2\\
	\leq & 2\|(W\otimes I_d -I_d)\big({\bf x}^k-\mathbf{\bar{x}}^k\big)-\eta (W\otimes I_d)\big({\bf y}^k-\mathbf{\bar{v}}^k\big)\|^2\\
	&+2\eta^2 n \|{ \bar{v}}^k\|^2\\
	\leq &{4\|{\bf x}^k-\mathbf{\bar{x}}^k\|^2+4\eta^2\rho^2\|{\bf y}^k-\mathbf{\bar{v}}^k\|^2}+2\eta^2n \|{ \bar{v}}^k\|^2.
	\end{align*}
	Taking the total expectation, we have
	\begin{align}\label{xplusx}
	&\mathbb{E} \|{\bf x}^{k+1}-{\bf x}^k\|^2\notag\\
	\overset{\eqref{vcondi}}{\leq} &{4\mathbb{E}\|{\bf x}^k-\mathbf{\bar{x}}^k\|^2+4\eta^2\rho^2\mathbb{E}\|{\bf y}^k-\mathbf{\bar{v}}^k\|^2}\notag\\
	&+2\eta^2n\Big(2\mathbb{E}\|\nabla f(\bar{x}^k)\|^2+\frac{2}{n}\sum_{i=1}^n(6L^2\mathbb{E}\|x_i^k-\bar{x}^k\|^2\notag\\
	&\qquad+4L^2\mathbb{E}\|\tau_i^k-\bar{x}^k\|^2)\Big)\notag\\
	\leq &{4\mathbb{E}\|{\bf x}^k-\mathbf{\bar{x}}^k\|^2+4\eta^2\rho^2\mathbb{E}\|{\bf y}^k-\mathbf{\bar{v}}^k\|^2}\notag\\
	&+4\eta^2n\mathbb{E}\|\nabla f(\bar{x}^k)\|^2+24\eta^2L^2\mathbb{E}\|{\bf x}^k-{\bf \bar{x}}^k\|^2\notag\\
	&+16\eta^2L^2\mathbb{E}\|\tau^k-{\bf \bar{x}}^k\|^2\notag\\
	= &\big(4+24\eta^2L^2\big)\mathbb{E}\|{\bf x}^k-\mathbf{\bar{x}}^k\|^2+4\eta^2\rho^2\mathbb{E} \|{\bf y}^k-\mathbf{\bar{v}}^k\|^2\notag\\
	&+4\eta^2n\mathbb{E}\|\nabla f(\bar{x}^k)\|^2+16\eta^2L^2\mathbb{E}\|\tau^k-\mathbf{\bar{x}}^k\|^2.
	\end{align}
	Hence, $\mathbb{E}\|{\bf y}^{k+1}-{\bf \bar{v}}^{k+1}\|^2$ satisfies
	\begin{align*}
	&\mathbb{E}\|{\bf y}^{k+1}-{\bf \bar{v}}^{k+1}\|^2\\
	\leq&2\rho^2 \mathbb{E}\|{\bf y}^k-{\bf \bar{v}}^k\|^2+2\rho^2E\|{\bf v}^{k+1}-{\bf v}^k\|^2\\
	\leq &2\rho^2 \mathbb{E}\|{\bf y}^k-{\bf \bar{v}}^k\|^2+2\rho^2\big[2L^2 {\mathbb{E}\|{\bf x}^{k+1}-{\bf x}^k\|^2}\\
	&+8L^2 {\mathbb{E}\|{\bf x}^{k+1}-\mathbf{\bar{x}}^{k+1}\|^2}+8L^2{\mathbb{E}\|\tau^{k+1}-\mathbf{\bar{x}}^{k+1}\|^2}\\
	&+8L^2 \mathbb{E} \|{\bf x}^k-\mathbf{\bar{x}}^k\|^2+8L^2 \mathbb{E}\|\tau^k-\mathbf{\bar{x}}^k\|^2\big]\\
	\overset{\eqref{xplusx}}{\leq} &2\rho^2\mathbb{E} \|{\bf y}^k-{\bf \bar{v}}^k\|^2\!+\!2\rho^2\!\Big[\!2L^2\!\Big(\!\big(\!4\!+24\eta^2L^2\big)\mathbb{E}\|{\bf x}^k-\mathbf{\bar{x}}^k\|^2\\
	&+4\eta^2\rho^2\mathbb{E}\|{\bf y}^k-\mathbf{\bar{v}}^k\|^2+4\eta^2n\mathbb{E}\|\nabla f(\bar{x}^k)\|^2\\
	&{+16\eta^2L^2\mathbb{E}\|\tau^k-\mathbf{\bar{x}}^k\|^2\Big)}+8L^2 {\mathbb{E}\|{\bf x}^{k+1}-\mathbf{\bar{x}}^{k+1}\|^2}\\
	&+8L^2{\mathbb{E}\|\tau^{k+1}-\mathbf{\bar{x}}^{k+1}\|^2}\\
	&+8L^2 \mathbb{E}\|{\bf x}^k-\mathbf{\bar{x}}^k\|^2+8L^2\mathbb{E}\|\tau^k-\mathbf{\bar{x}}^k\|^2\Big]\\
	\overset{\eqref{xsubxbareq},\eqref{tausubxbar}}{\leq} &2\rho^2\mathbb{E} \|{\bf y}^k-{\bf \bar{v}}^k\|^2\!+\!2\rho^2\!\Bigg[\!2L^2\!\Big(\!\big(\!4\!+24\eta^2L^2\big)\mathbb{E}\|{\bf x}^k-\mathbf{\bar{x}}^k\|^2\\
	&+4\eta^2\rho^2\mathbb{E}\|{\bf y}^k-\mathbf{\bar{v}}^k\|^2+4\eta^2n\mathbb{E}\|\nabla f(\bar{x}^k)\|^2\\
	&{+16\eta^2L^2\mathbb{E}\|\tau^k-\mathbf{\bar{x}}^k\|^2}\Big)\\
	&+8L^2\bigg\{ {\Big[2\rho^2\mathbb{E}\|{\bf x}^k-{\bf \bar{x}}^k\|^2+2\rho^2\eta^2\mathbb{E}\|{\bf y}^k-{\bf \bar{v}}^k\|^2\Big]}\\
	&+\Big[(1-P)\big((\eta^2+\frac{\eta}{\beta})8L^2+(1+\eta\beta)\big)\mathbb{E} \|\tau^k-{\bf \bar{x}}^k\|^2\\
	&+\!(\eta^2\!+\!\frac{\eta}{\beta})2n (\!1\!-\!P)\mathbb{E}\|\nabla\! f(\bar{x}^k\!)\|^2\!+\! 2\rho^2\eta^2 P \mathbb{E}\|{\bf y}^k\!-\!{\bf \bar{v}}^k\!\|^2\\
	&+\big((1-P) (\eta^2+\frac{\eta}{\beta})12 L^2+ 2\rho^2 P \big)\mathbb{E}\|{\bf x}^k-{\bf \bar{x}}^k\|^2\Big]\bigg\}\\
	&+8L^2 \mathbb{E}\|{\bf x}^k-\mathbf{\bar{x}}^k\|^2+8L^2\mathbb{E}\|\tau^k-\mathbf{\bar{x}}^k\|^2\Bigg]\\
	\leq & (2\rho^2+48\eta^2L^2\rho^4+32 P \eta^2L^2\rho^4) \mathbb{E}\|{\bf y}^k-{\bf \bar{v}}^k\|^2\\
	&+16\rho^2L^2\big(1+6\eta^2L^2+(2+2P)\rho^2\notag\\
	&\qquad+12L^2(1-P)(\eta^2+\frac{\eta}{\beta})\big)\mathbb{E}[\|{\bf x}^k-{\bf \bar{x}}^k\|^2]\notag\\
	&+16\rho^2L^2\big(4\eta^2L^2+(1-P)[1+\eta\beta + (\eta^2+\frac{\eta}{\beta})8L^2]\big)\notag\\
	&\qquad\mathbb{E}[\|\tau^k-{\bf \bar{x}}^k\|^2]\notag\\
	&+16\rho^2L^2n\big(\eta^2+2 (1-P)(\eta^2+\frac{\eta}{\beta})\big)  \mathbb{E}[\|\nabla f(\bar{x}^k)\|^2].	
	\end{align*}
Now, with Assumption \ref{W_assump}, we have proved all inequalities in Proposition \ref{xsub}.
\end{proof}

\bibliographystyle{ieeetran}
\bibliography{refer}

\end{document}